\documentclass[a4paper,12pt]{amsart}
\usepackage[latin1]{inputenc}
\usepackage[T1]{fontenc}
\usepackage{amsmath}
\usepackage{amssymb}
\usepackage{amsthm}

\usepackage[top=20mm, bottom=20mm, left=25mm, right=25mm]{geometry}

\usepackage{appendix} 
\usepackage{enumitem} 
\usepackage{fnpct} 
\usepackage{graphicx} 
\usepackage{hyperref} 

\theoremstyle{theorem}

\newtheorem{theorem}{Theorem}[section]
\newtheorem{proposition}{Proposition}[section]
\newtheorem{lemma}{Lemma}[section]
\newtheorem{claim}{Claim}[section]
\newtheorem{corollary}{Corollary}[section]
\theoremstyle{definition}
\newtheorem{definition}{Definition}[section]
\newtheorem{remark}{Remark}[section]
\newtheorem{problem}{Problem}[section]

\newcommand{\Q}{\ensuremath{\mathbb{Q}}}
\newcommand{\R}{\ensuremath{\mathbb{R}}}
\newcommand{\Z}{\ensuremath{\mathbb{Z}}}
\newcommand{\N}{\ensuremath{\mathbb{N}}}
\newcommand{\M}{\mathrm{M}}
\newcommand{\abs}[1]{\left\vert #1 \right\vert}
\newcommand{\norme}[1]{\left\Vert #1 \right\Vert}
\renewcommand{\span}{\mathrm{Span}}
\newcommand{\muexp}[4]{\mathring\mu_{#1}(#2\vert #3)_{#4}}
\newcommand{\muexpA}[4]{\mu_{#1}(#2\vert #3)_{#4}}
\renewcommand{\epsilon}{\varepsilon}
\renewcommand{\le}{\leqslant}
\renewcommand{\ge}{\geqslant}
\newcommand{\transp}{\,{}^t\!}
\newcommand{\fonction}[5]{\begin{array}{lrcl}
#1\colon & #2 & \longrightarrow & #3 \\
    & #4 & \longmapsto &  #5  \end{array}}

\title{On the approximation exponents for subspaces of $\R^n$}
\date{\today}
\author[E. Joseph]{Elio Joseph}
\address{Universit\'e Paris-Saclay, CNRS, Laboratoire de math\'ematiques d'Orsay, 91405, Orsay, France.}

\email{josephelio@gmail.com}

\begin{document}

	\begin{abstract}
	This paper follows the generalisation of the classical theory of Diophantine approximation to subspaces of $\R^n$ established by W. M. Schmidt in 1967. Let $A$ and $B$ be two subspaces of $\mathbb R^n$ of respective dimensions $d$ and $e$ with $d+e\leqslant n$. The proximity between $A$ and $B$ is measured by $t=\min(d,e)$ canonical angles $0\leqslant \theta_1\leqslant \cdots\leqslant \theta_t\leqslant \pi/2$; we set $\psi_j(A,B)=\sin\theta_j$. If $B$ is a rational subspace, his complexity is measured by its height $H(B)=\mathrm{covol}(B\cap\mathbb Z^n)$. We denote by $\muexpA nAej$ the exponent of approximation defined as the upper bound (possibly equal to $+\infty$) of the set of $\beta>0$ such that the inequality $\psi_j(A,B)\leqslant H(B)^{-\beta}$ holds for infinitely many rational subspaces $B$ of dimension $e$. We are interested in the minimal value $\muexp ndej$ taken by $\muexpA nAej$ when $A$ ranges through the set of subspaces of dimension $d$ of $\R^n$ such that for all rational subspaces $B$ of dimension $e$ one has $\dim (A\cap B)<j$. We show that $\muexp 4221=3$, $\muexp 5321\le 6$ and $\muexp {2d}d\ell1\le 2d^2/(2d-\ell)$. We also prove a lower bound in the general case, which implies that $\muexp nddd\xrightarrow[n\to+\infty]{} 1/d$.
	\end{abstract}

\maketitle

\section{Introduction}

The classical theory of Diophantine approximation studies how well points of $\R^n$ can be approximated by rational points. Here, we are interested in a problem studied by W. M. Schmidt in 1967 (see \cite{schmidt67}), which consists in approximating subspaces of $\R^n$ by rational subspaces. The results presented here can be found in my Ph.D. thesis (see \cite{joseph21} chapters 3 and 4 for more details).

A subspace of $\R^n$ is said to be \emph{rational} whenever it admits a basis of vectors with rational coordinates. Denote by $\mathfrak R_n(e)$ the set of rational subspaces of dimension $e$ of $\R^n$. A subspace $A$ of $\R^n$ is called \emph{$(e,j)$-irrational} whenever for all $B\in\mathfrak R_n(e)$, $\dim(A\cap B)<j$; notice that being $(e,1)$-irrational is equivalent to intersecting trivially all subspaces of $\mathfrak R_n(e)$. Denote by $\mathfrak I_n(d,e)_j$ the set of all $(e,j)$-irrational subspaces of dimension $d$ of $\R^n$.

Let us define a notion of \emph{complexity} for a rational subspace and a notion of \emph{proximity} between two subspaces, which will lead to the formulation of the main problem.

Let $B\in\mathfrak R_n(e)$; one can choose $\Xi\in\Z^N$, with $N=\binom ne$, a vector with setwise coprime coordinates in the class of Pl\"ucker coordinates of $B$. Let us define the \emph{height} of $B$ to be the Euclidean norm of $\Xi$:
\[H(B)=\norme\Xi.\]

Endow $\R^n$ with the standard Euclidean norm, and define the distance between two vectors $X,Y\in\R^n\setminus\{0\}$ by
\[\psi(X,Y)=\sin\widehat{(X,Y)}=\frac{\norme{X\wedge Y}}{\norme{X}\cdot \norme{Y}}\]
where $X\wedge Y$ is the exterior product of $X$ and $Y$, and the Euclidean norm $\norme\cdot$ is naturally extended to $\Lambda^2(\R^n)$ so that $\norme{X\wedge Y}$ is the area of the parallelogram spanned by $X$ and $Y$. Let $A$ and $B$ be two subspaces of $\R^n$ of dimensions $d$ and $e$ respectively. One can define by induction $t=\min(d,e)$ angles between $A$ and $B$. Let us define
\[\psi_1(A,B)=\min_{\substack{X\in A\setminus\{0\}\\Y\in B\setminus\{0\}}}\psi(X,Y)\]
and denote by $X_1$ and $Y_1$ unitary vectors such that $\psi(X_1,Y_1)=\psi_1(A,B)$. Then, by induction, it is assumed that $\psi_1(A,B),\ldots,\psi_j(A,B)$ have been constructed for $j\in\{1,\ldots,t-1\}$, associated with couples of vectors $(X_1,Y_1),\ldots,(X_j,Y_j)\in A\times B$ respectively. One denotes by $A_j$ the orthogonal complement of $\span(X_1,\ldots,X_j)$ in $A$ and by $B_j$ the orthogonal complement of $\span(Y_1,\ldots,Y_j)$ in $B$. Let us define in a similar fashion 
\[\psi_{j+1}(A,B)=\min_{\substack{X\in A_j\setminus\{0\}\\Y\in B_j\setminus\{0\}}}\psi(X,Y)\]
and denote by $X_{j+1}$ and $Y_{j+1}$ unitary vectors such that $\psi(X_{j+1},Y_{j+1})=\psi_{j+1}(A,B)$. 

These angles between $A$ and $B$ are canonical in the sense of this paragraph, based on \cite{schmidt67}, Theorem 4. This will also be used to prove Claim \ref{minorationinverseavecdroitesbienchoisies_paeivmsoivn} in Section \ref{section_lowerbound_general_paifnvpvisnv} below. There exist orthonormal bases $(X_1,\ldots,X_d)$ and $(Y_1,\ldots,Y_e)$ of $A$ and $B$ respectively, and real numbers $0\le \theta_t\le \cdots\le \theta_1\le 1$ such that for all $i\in\{1,\ldots,d\}$ and for all $j\in\{1,\ldots,e\}$, $X_i\cdot Y_j=\delta_{i,j}\cos\theta_i$ where $\delta$ is the Kronecker delta and $\cdot$ is the canonical scalar product on $\R^n$. Moreover, the numbers $\theta_1,\ldots,\theta_t$ are independent of the bases $(X_1,\ldots,X_d)$ and $(Y_1,\ldots,Y_e)$ chosen. Notice that $\psi_j(A,B)=\sin\theta_j$.

We can now formulate the main problem. Let $n\ge 2$, $d,e\in\{1,\ldots,n-1\}$ such that $d+e\le n$, $j\in\{1,\ldots,\min(d,e)\}$, and $A\in\mathfrak I_n(d,e)_j$. Let us define by $\muexpA nAej$ the upper bound (possibly equal to $+\infty$) of all $\beta>0$ such that
\[\psi_j(A,B)\le \frac 1{H(B)^\beta}\]
holds for infinitely many $B\in\mathfrak R_n(e)$. One also defines 
\[\muexp ndej=\inf_{A\in\mathfrak I_n(d,e)_j}\muexpA nAej.\]
\begin{problem}\label{problem_main_bamoeribvob}
Determine $\muexp ndej$ in terms of $n,d,e,j$.
\end{problem}
Schmidt proved several bounds on the quantity $\muexp ndej$ in 1967 (see \cite{schmidt67}, Theorems 12, 13, 15, 16 and 17). In all what follows, let $t=\min(d,e)$.
\begin{theorem}[Schmidt, 1967]\label{encadrement_general_schmidt_aoeibinvn}
For all $j\in\{1,\ldots,t\}$, one has
\[\frac{d(n-j)}{j(n-d)(n-e)}\le \muexp ndej\le \frac 1j\left\lceil\frac{e(n-e)+1}{n+1-d-e}\right\rceil,\]
moreover, when $j=1$:
\[\muexp nde1\ge\frac{n(n-1)}{(n-d)(n-e)}.\]
\end{theorem}
Schmidt improved the lower bound when an additional hypothesis is met. He also determined some exact values of $\muexp ndej$. In particular, Problem \ref{problem_main_bamoeribvob} is completely solved when $\min(d,e)=1$.
\begin{theorem}[Schmidt, 1967]\label{deuxieme_borne_Schmidt_reuhfgezliuhl}
Let $j\in\{1,\ldots,t\}$. If
\[j+n-t\ge j(j+n-d-e),\]
then
\[\muexp ndej\ge \frac{j+n-t}{j(j+n-d-e)},\]
moreover, when $j=t$:
\[\muexp ndet= \frac{n}{t(t+n-d-e)}.\]
\end{theorem}
A direct application of Schmidt's Going-up theorem (\cite{schmidt67}, Theorem 9) is the following result proved in section \ref{section_corollaire_Moshchevitin_peinvipsnvpnp} below.
\begin{proposition}\label{prop_app_goingup_murond_oebonio}
Let $d,e,j,\ell\in\N^*$ be such that $d+e\le n$, $1\le j\le \ell \le e$ and $j\le d$. Then
\[\muexp ndej\ge \frac{n-\ell}{n-e}\cdot \muexp nd\ell j.\]
\end{proposition}
This proposition implies some straightforward improvements. For instance, the known lower bound $\muexp 6332\ge 5/4$ (Theorem \ref{encadrement_general_schmidt_aoeibinvn}) becomes $\muexp 6332\ge 4/3$ using $\muexp 6322=1$ (Theorem \ref{deuxieme_borne_Schmidt_reuhfgezliuhl}).

In 2020, both N. Moshchevitin (\cite{moshchevitin20}, Satz 2) and N. de Saxcé (\cite{saxce20}, Theorem 9.3.2) improved some upper bounds.
\begin{theorem}[Moshchevitin, 2020]\label{theoreme_moshchevitin_abenobvin}
Let $d\ge 1$ be an integer, one has
\[\muexp {2d}dd1\le 2d.\]
\end{theorem}
\begin{theorem}[Saxcé, 2020]\label{theoreme_saxcé_piebpisnvpzin}
Let $n\ge 2$ and $d\in\{1,\ldots,\lfloor n/2\rfloor\}$. One has
\[\muexp nddd\le \frac n{d(n-d)}.\]
\end{theorem}
The simplest unknown case and also the last unknown case in $\R^4$ is $(n,d,e,j)=(4,2,2,1)$. Theorem \ref{encadrement_general_schmidt_aoeibinvn} together with Theorem \ref{theoreme_moshchevitin_abenobvin} gives $3\le\muexp 4221\le 4$. Here, we will show the following theorem.
\begin{theorem}\label{cas_4221egal3_abomeneabfjbnv}
One has
\[\muexp 4221=3.\]
\end{theorem}
The next unknown cases are in $\R^5$. One can notice that Theorem \ref{encadrement_general_schmidt_aoeibinvn} combined with Theorem \ref{deuxieme_borne_Schmidt_reuhfgezliuhl} give $4\le\muexp 5321\le 7$. This upper bound is improved by $1$.
\begin{theorem}\label{theoreme_central_R5_baneiofbonvnss}
One has
\[\muexp 5321\le 6.\]
\end{theorem}
Combining Theorem \ref{theoreme_moshchevitin_abenobvin} and Proposition \ref{prop_app_goingup_murond_oebonio}, an improvement on the known bound for $\muexp{2d}d{\ell}1$ is deduced; see the beginning of section \ref{section_corollaire_Moshchevitin_peinvipsnvpnp} for examples.
\begin{theorem}\label{theoreme_avec_moshchevitin_goingup_aeronfvn}
Let $d\ge 2$ and $\ell\in\{1,\ldots,d\}$, one has
\[\muexp {2d}d{\ell}1\le \frac{2d^2}{2d-\ell}.\]
\end{theorem}
Finally, we prove a new lower bound in the general case.
\begin{theorem}\label{premiereborneobtenue_amorimeovbn}
Let $n\ge 4$ and $d,e\in\{1,\ldots,n-1\}$ such that $d+e\le n$; let $j\in\{1,\ldots,\min(d,e)\}$. One has
\[\muexp ndej\ge \frac{(n-j)(jn-jd+j^2/2+j/2+1)}{j^2(n-e)(n-d+j/2+1/2)}.\]
\end{theorem}
This leads to the following corollary. 
\begin{corollary}\label{corollaire_th_lower_bound_etdeSaxce_zpiznvpsinvc}
One has, for any fixed $d\ge 1$:
\[\lim_{n\to+\infty} \muexp nddd = \frac 1d.\]
\end{corollary}

Section \ref{section_R4_boaeinvoin} focuses on the case of the approximation of a plane by rational planes in $\R^4$ (Theorem \ref{cas_4221egal3_abomeneabfjbnv}). In Section \ref{section_R5_epribnrfv} we approximate a subspace of dimension $3$ by rational planes (Theorem \ref{theoreme_central_R5_baneiofbonvnss}). Then, in Section \ref{section_comments_vaoebnisfon}, we comment briefly on the method developed in the previous two sections. Section \ref{section_corollaire_Moshchevitin_peinvipsnvpnp} contains a proof of Theorem \ref{theoreme_avec_moshchevitin_goingup_aeronfvn}. Finally, Section \ref{section_lowerbound_general_paifnvpvisnv} develops how to decompose the subspace one wants to approach into subspaces of lower dimensions, and this leads to a proof of Theorem \ref{premiereborneobtenue_amorimeovbn} and Corollary \ref{corollaire_th_lower_bound_etdeSaxce_zpiznvpsinvc}.

\section{Approximation of a plane by rational planes in $\mathbb R^4$}\label{section_R4_boaeinvoin}

The main result is Theorem \ref{cas_4221egal3_abomeneabfjbnv}: $\muexp 4221=3$. It finishes to solve Problem \ref{problem_main_bamoeribvob} for $n\le 4$. To prove this theorem, some planes of $\R^4$ are explicitly constructed, which are $(2,1)$-irrational and not so well approximated by rational planes. For $\xi\in]0,\sqrt 7[$, let us consider the plane $A_\xi$ of $\R^4$ spanned by
\[X^{(1)}_\xi=\begin{pmatrix} 0 \\1 \\\xi \\\sqrt{7-\xi^2}\end{pmatrix}\quad \text{ and }\quad X^{(2)}_\xi=\quad\begin{pmatrix} 1 \\0\\-\sqrt{7-\xi^2}\\\xi\end{pmatrix}.\]
The crucial lemma in order to prove Theorem \ref{cas_4221egal3_abomeneabfjbnv} is Lemma \ref{condition_dirrationalite_4221_regouzhelfgu} below, which requires the following function $\varphi$:
\begin{equation}\label{eqphi_mpairhbnobiv}
\varphi(A,B)=\prod_{j=1}^{\min(\dim A,\dim B)} \psi_j(A,B).
\end{equation}
\begin{lemma}\label{condition_dirrationalite_4221_regouzhelfgu}
There exist real numbers $\xi\in]0,\sqrt 7[$ and $c>0$ such that $A_\xi\in\mathfrak I_4(2,2)_1$ and
\begin{equation}\label{minoration_hauteur_lemme_R4_zidovdbon}
\forall B\in\mathfrak R_4(2),\quad \varphi(A_\xi,B)\ge\frac c{H(B)^{3}}.
\end{equation}
\end{lemma}
From Lemma \ref{condition_dirrationalite_4221_regouzhelfgu} and Lemma \ref{minoration_psiphi_eomivocvbn} below, we shall deduce the following proposition.
\begin{proposition}\label{proposition_R4_mu_Axi_zgiozovdib}
There exists $\xi\in]0,\sqrt 7[$ such that
\[\muexpA 4{A_\xi}21=3.\]
\end{proposition}
Theorem \ref{cas_4221egal3_abomeneabfjbnv} comes directly from the definition of $\mathring\mu$, Proposition \ref{proposition_R4_mu_Axi_zgiozovdib} and Theorem \ref{encadrement_general_schmidt_aoeibinvn}. Before, proving Proposition \ref{proposition_R4_mu_Axi_zgiozovdib}, let us introduce some notations and two basic lemmas. 

Given vectors $X_1,\ldots,X_e\in\R^n$, let us denote by $M\in\M_{n,e}(\R)$ the matrix whose $j$-th column is $X_j$ for $j\in\{1,\ldots,e\}$. Let us define the \emph{generalised determinant} of the family $(X_1,\ldots,X_e)$ to be $D(X_1,\ldots,X_e)=\sqrt{\det(\transp MM)}$. The following result gives an equivalent definition of the height of a rational subspace (see \cite{schmidt67}, Theorem 1).
\begin{theorem}\label{th_def_equiv_hauteur_vaoribibgipn}
Let $B\in\mathfrak R_n(e)$ and $(X_1,\ldots,X_e)$ be a basis of $B\cap\Z^n$. Then 
\[H(B)=D(X_1,\ldots,X_e).\]
\end{theorem}
Let us make a link between proximity and height.
\begin{lemma}\label{lien_proximite_hauteur_zozofvbcz}
Let $n\ge 2$, $d,e\in\{1,\ldots,n-1\}$ be such that $d+e=n$, $A$ be a subspace of dimension $d$ of $\R^n$ and $B\in\mathfrak R_n(e)$. Let $(X_1,\ldots,X_d)$ be a basis of $A$, $(Y_1,\ldots,Y_e)$ be a basis of $B\cap \Z^n$, and denote by $M\in\M_n(\R)$ the matrix whose columns are $X_1,\ldots,X_d,Y_1,\ldots,Y_e$ respectively. There exists a constant $c>0$ depending only on $(X_1,\ldots,X_d)$ such that
\[\varphi(A,B)=c\ \frac{\abs{\det M}}{H(B)}.\]
\end{lemma}
\begin{proof}
The following claim comes from equation (7) page 446 of \cite{schmidt67}.
\begin{claim}\label{claim_proximity_avec_det_gene_aneoribsodiwb}
One has
\[\varphi(A,B)=\frac{D(X_1,\ldots,X_d,Y_1,\ldots,Y_e)}{D(X_1,\ldots,X_d)D(Y_1,\ldots,Y_e)}.\]
\end{claim}
Since $(Y_1,\ldots,Y_e)$ is a basis of $B\cap\Z^n$, Claim \ref{claim_proximity_avec_det_gene_aneoribsodiwb} together with Theorem \ref{th_def_equiv_hauteur_vaoribibgipn} gives us 
\[\varphi(A,B)=cD(X_1,\ldots,X_d,Y_1,\ldots,Y_e)H(B)^{-1}\]
where $c=D(X_1,\ldots,X_d)^{-1}>0$ is a constant depending only on $(X_1,\ldots,X_d)$. Moreover, the matrix $M$ is a square matrix, so $D(X_1,\ldots,X_d,Y_1,\ldots,Y_e)^2=\det(\transp MM)=\det (M)^2$. Thereby, since $D(X_1,\ldots,X_d,Y_1,\ldots,Y_e)\ge 0$, one has $\varphi(A,B)=c\abs{\det M}H(B)^{-1}$.
\end{proof} 
\begin{lemma}\label{minoration_psiphi_eomivocvbn}
Let $n\ge 2$, $A$ and $B$ be two subspaces of $\R^n$ of dimensions $d$ et $e$ respectively. Then for all $j\in\{1,\ldots,\min(d,e)\}$, $\psi_j(A,B)\ge \varphi(A,B)^{1/j}$.
\end{lemma}
\begin{proof}
Let $t=\min(d,e)$ and $j\in\{1,\ldots,t\}$. From the definition of the $\psi_i$, one has $\psi_1(A,B)\le\cdots\le \psi_t(A,B)\le 1$. Thereby, the product in Equation \eqref{eqphi_mpairhbnobiv} can be split in this way:
\[\varphi(A,B)=\left(\prod_{i=1}^{j}\underbrace{\psi_i(A,B)}_{\le\psi_j(A,B)}\right)\times\left(\prod_{i=j+1}^t\underbrace{\psi_i(A,B)}_{\le 1}\right)\le \psi_j(A,B)^{j}.\]
\end{proof}

We can now provide a proof of Proposition \ref{proposition_R4_mu_Axi_zgiozovdib}.
\begin{proof}[Proof of Proposition \ref{proposition_R4_mu_Axi_zgiozovdib}.]
Together with Lemma \ref{minoration_psiphi_eomivocvbn} applied for $j=1$, Lemma \ref{condition_dirrationalite_4221_regouzhelfgu} shows that $\muexpA 4{A_\xi}21\le 3$. Since Theorem \ref{encadrement_general_schmidt_aoeibinvn} gives $\muexpA 4{A_\xi}21\ge \muexp 4221\ge 3$, Proposition \ref{proposition_R4_mu_Axi_zgiozovdib} follows.
\end{proof}

We will prove a final lemma before tackling the proof of the central Lemma \ref{condition_dirrationalite_4221_regouzhelfgu}.
\begin{lemma}\label{lemme_pour_coordPluck_premsentreelles_egouzbelvc}
Let $n\ge 2$ be an integer, $e\in\{1,\ldots,n\}$ and $B\in\mathfrak R_n(e)$. There exists a basis $(X_1,\ldots,X_e)$ of $B\cap\Z^n$ such that if one denotes by $\eta=(\eta_1,\ldots,\eta_N)$, where $N=\binom ne$, the Pl\"ucker coordinates associated with $(X_1,\ldots,X_e)$ and ordered by lexicographic order, one has $\eta\in\Z^N$ and $\gcd(\eta_1,\ldots,\eta_N)=1$.
\end{lemma}
\begin{proof}
Since $B$ is a rational subspace, $B\cap\Z^n$ is a $\Z$-submodule of the free $\Z$-module $\Z^n$. According to the structure theorem for finitely generated modules over a principal ideal domain, there exist a basis $(X_1,\ldots,X_n)$ of $\Z^n$ and integers $d_1,\ldots,d_e\ge 1$ such that $(d_1X_1,\ldots,d_eX_e)$ is a basis of $B\cap\Z^n$. Let $i\in\{1,\ldots,e\}$; since $d_iX_i\in B\cap\Z^n$ and $X_i\in\Z^n$, one has $X_i\in B\cap\Z^n$, therefore $d_i=1$, so $(X_1,\ldots,X_e)$ is a basis of $B\cap\Z^n$.

Let us denote by $M$ the matrix of $\M_n(\Z)$ whose columns are $X_1,\ldots,X_n$ respectively. Let us also denote by $M_1$ the matrix of $\M_{n,e}(\Z)$ formed with the $e$ first columns of $M$ and by $M_2$ the matrix of $M_{n,n-e}(\Z)$ formed with the $n-e$ last columns of $M$. Notice that the minors of size $e\times e$ of $M_1$ ordered by lexicographic order give an element $(\eta_1,\ldots,\eta_N)\in\Z^N$ of the class of Pl\"ucker coordinates of $B$ associated with the basis $(X_1,\ldots,X_e)$. Let us denote by $\delta_1,\ldots,\delta_N$ the minors of size $(n-e)\times (n-e)$ of $M_2$ ordered by lexicographic order. Computing the determinant of $M$ using a Laplace expansion on its $e$ firsts columns gives
\begin{equation}\label{equation_Bezout_generalise_vaeoneobnvsvnos}
\abs{\det M}=\abs{\sum_{i=1}^N \epsilon(i)\eta_i\delta_{N+1-i}}=1
\end{equation}
because $\abs{\det M}=\mathrm{covol}(\Z^n)=1$, where $\epsilon$ is a function with values in $\{\pm1\}$. Since \eqref{equation_Bezout_generalise_vaeoneobnvsvnos} is a generalised Bézout identity, one can conclude that $\gcd(\eta_1,\ldots,\eta_N)=1$.
\end{proof}
In order to prove Lemma \ref{condition_dirrationalite_4221_regouzhelfgu}, we will use the following definition and theorem (see \cite{beresnevich15}, Corollary 1).
\begin{definition}
Let $\mathbf{Bad}$ be the set of all $y\in\R^k$ such that there exists $c>0$ such that the only integer solution $(a_0,\ldots,a_k)$ to the inequality 
\[\abs{a_0+a_1y_1+\cdots+a_ky_k}<c\norme{(a_1,\ldots,a_k)}_{\infty}^{-k}\]
is the trivial one $(0,\ldots,0)$.
\end{definition}
\begin{theorem}[Beresnevich, 2015]\label{th_bernesnevich_bad_pabeinbfpisnv}
Let $\mathcal M$ be a manifold immersed into $\R^n$ by an analytic nondegenerate map. Then $\mathbf{Bad}\cap \mathcal M$ has the same Hausdorff dimension as $\mathcal M$; in particular $\mathbf{Bad}\cap\mathcal M\ne \emptyset$.
\end{theorem}
Finally, let us prove Lemma \ref{condition_dirrationalite_4221_regouzhelfgu}.
\begin{proof}[Proof of Lemma \ref{condition_dirrationalite_4221_regouzhelfgu}.]
Let $B\in\mathfrak R_4(2)$ and $(Y_1,Y_2)$ be a basis of $B$ provided by Lemma \ref{lemme_pour_coordPluck_premsentreelles_egouzbelvc}. Let us denote by $(\eta_1,\ldots,\eta_6)$ a set of Pl\"ucker coordinates of $B$ associated with the basis $(Y_1,Y_2)$ as in Lemma \ref{lemme_pour_coordPluck_premsentreelles_egouzbelvc}, so that $(\eta_1,\ldots,\eta_6)\in\Z^6$ and $\gcd(\eta_1,\ldots,\eta_6)=1$. Moreover, this vector satisfies the Pl\"ucker relation (see \cite{caldero15}, Theorem 2.9) for a subspace of dimension $2$ of $\R^4$:
\begin{equation}\label{relationPluck_R4_oehvodbc}
\eta_1\eta_6-\eta_2\eta_5+\eta_3\eta_4=0.
\end{equation}
The manifold $\mathcal M=\{(1,\xi,\sqrt{7-\xi^2}),\ \xi\in]0,\sqrt 7[\}$ is nondegenerate (the functions $\xi\mapsto 1$, $\xi\mapsto\xi$, and $\xi\mapsto\sqrt{7-\xi^2}$ are linearly independent over $\R$), so Theorem \ref{th_bernesnevich_bad_pabeinbfpisnv} implies the existence of $\xi\in]0,\sqrt 7[$ such that $(1,\xi,\sqrt{7-\xi^2})\in\mathbf{Bad}$. In particular $1$, $\xi$ and $\sqrt{7-\xi^2}$ are linearly independent over $\Q$. Let us denote by $M_\xi$ the matrix of $\M_4(\R)$ whose columns are $X_\xi^{(1)},X_\xi^{(2)},Y_1,Y_2$ respectively. Notice that $A_\xi\cap B=\{0\}$ if, and only if, $\det M_\xi\ne 0$. The determinant of $M_\xi$ is computed by a Laplace expansion on its two first columns:
\begin{equation}\label{valeur_detM_R4_zghodidvc}
\det M_\xi=-\eta_6+\eta_5\xi-\eta_4\sqrt{7-\xi^2}-\eta_3\sqrt{7-\xi^2}-\eta_2\xi+7\eta_1.
\end{equation}
Assuming that $\det M_\xi=0$ implies
\begin{equation}\label{coord_Pluck_egales0_R4_ziogozivnb}
-\eta_6+7\eta_1+(\eta_5-\eta_2)\xi+(-\eta_3-\eta_4)\sqrt{7-\xi^2}=0. 
\end{equation}
Since $\dim_\Q\span_\Q(1,\xi,\sqrt{7-\xi^2})=3$ and the $\eta_i$ are integers, Equation \eqref{coord_Pluck_egales0_R4_ziogozivnb} gives
\begin{equation}\label{cases_R4_peariohoeiv}
(\eta_4,\eta_5,\eta_6)=(-\eta_3,\eta_2,7\eta_1).
\end{equation}
Thereby, Equality \eqref{relationPluck_R4_oehvodbc} becomes
\[\eta_2^2+\eta_3^2=7\eta_1^2.\]
Reducing modulo $4$, this equation implies that $\eta_1$, $\eta_2$ and $\eta_3$ are even, which contradicts the assumption $\gcd(\eta_1,\ldots,\eta_6)=1$ using Equation \eqref{cases_R4_peariohoeiv}. Thereby, $\det M_\xi\ne 0$, so $A_\xi\cap B=\{0\}$ which proves that the subspace $A_\xi$ is $(2,1)$-irrational. \\

To establish Inequality \eqref{minoration_hauteur_lemme_R4_zidovdbon} of Lemma \ref{condition_dirrationalite_4221_regouzhelfgu}, notice that the basis $(Y_1,Y_2)$ of $B$ is provided by Lemma \ref{lemme_pour_coordPluck_premsentreelles_egouzbelvc}, so it is also a $\Z$-basis of $B\cap \Z^4$. Hence, Lemma \ref{lien_proximite_hauteur_zozofvbcz} gives a constant $c_1>0$ depending only on $(X_\xi^{(1)},X_\xi^{(2)})$, such that
\begin{equation}\label{egalite_phi_R4_peiohbeofivn}
\varphi(A_\xi,B)=\abs{\det(M_\xi)}\frac{c_1}{H(B)}.
\end{equation}
Since the Pl\"ucker coordinates $\eta=(\eta_1,\ldots,\eta_6)$ of $B$ are integers and satisfy $\gcd(\eta_1,\ldots,\eta_6)=1$, one has
\begin{equation}\label{R4_hauteur_fonction_de_eta_iahrnbodianv}
H(B)=\norme \eta.
\end{equation}
Now recall that we have chosen $\xi$ in such a way that there exists a constant $c_2>0$ such that for all $q=(a,b,c)\in\Z^3\setminus\{(0,0,0)\}$:
\begin{equation}\label{minorationKB_R4_orubvdobv}
\abs{a\sqrt{7-\xi^2}+b\xi+c}\ge c_2\norme{q}^{-2}.
\end{equation}
Notice that for $q=(-\eta_3-\eta_4,\eta_5-\eta_2,-\eta_6+7\eta_1)$, one has $q\ne (0,0,0)$ otherwise \eqref{cases_R4_peariohoeiv} would be true, and it was already said that this was impossible. Moreover, $\norme{q}\le \sqrt{67}\norme{\eta}$, so Inequality \eqref{minorationKB_R4_orubvdobv} combined with Equality \eqref{valeur_detM_R4_zghodidvc} gives
\[\abs{\det(M_\xi)}\ge c_3\norme{\eta}^{-2}.\]
This inequality together with \eqref{egalite_phi_R4_peiohbeofivn} and \eqref{R4_hauteur_fonction_de_eta_iahrnbodianv} give a constant $c_4>0$ such that
\[\varphi(A_\xi,B)\ge \frac{c_4}{H(B)^{3}}.\]
\end{proof}
\begin{remark}
In the same way, one can construct infinitely many subspaces $A_\xi$ defined over $\overline\Q$ satisfying $\muexpA 4{A_\xi}21=3$ with a theorem of Schmidt. The point is to replace in the proof of Lemma \ref{condition_dirrationalite_4221_regouzhelfgu} the use of Theorem \ref{th_bernesnevich_bad_pabeinbfpisnv} by Theorem 2 of \cite{schmidt70}; the only difference is that the exponent $-2$ in Equation \eqref{minorationKB_R4_orubvdobv} becomes $-2-\epsilon$ for any $\epsilon>0$, and $-3$ becomes $-3-\epsilon$ in Equation \eqref{minoration_hauteur_lemme_R4_zidovdbon}. Up to this modification, Lemma \ref{condition_dirrationalite_4221_regouzhelfgu} and Proposition \ref{proposition_R4_mu_Axi_zgiozovdib} are still true if $\xi\in]0,\sqrt 7[$ is a real algebraic number satisfying $\dim_\Q\span_\Q(1,\xi,\sqrt{7-\xi^2})=3$. In particular, for $\xi=\sqrt 2$, one gets the explicit example
\[\muexpA 4{A_{\sqrt 2}}21=3.\]
\end{remark}

\section{Approximation of a subspace of dimension $3$ by rational planes in $\R^5$}\label{section_R5_epribnrfv}

The method developed here is very similar to the one used in Section \ref{section_R4_boaeinvoin}, so we will not linger on the details in this section. Computations are not detailed, see \cite{joseph21} for extended computations. The main result is Theorem \ref{theoreme_central_R5_baneiofbonvnss}: $\muexp 5321\le 6$.

As in Section \ref{section_R4_boaeinvoin}, a subspace of $\R^5$ is explicitly constructed so that it is $(2,1)$-irrational and at the same time not so well approximated by rational planes of $\R^5$. We will start by stating some lemmas to prove this statement; the proofs of the lemmas will follow later.

Let $\zeta_3$ be a real number, let us consider the four real numbers:
\[\zeta_1=-\frac{112 \, \zeta_{3}^{4} - 196 \, \zeta_{3}^{3} - {\left(42 \, \sqrt{2} \zeta_{3}^{3} - 17 \, \sqrt{2} \zeta_{3}^{2} + 13 \, \sqrt{2} \zeta_{3}\right)} \sqrt{4 \, \zeta_{3} - 5} \sqrt{\zeta_{3} - 1} + 88 \, \zeta_{3}^{2} - 30 \, \zeta_{3} + 6}{4 \, {\left(10 \, \zeta_{3}^{4} - 7 \, \zeta_{3}^{3} - {\left(4 \, \sqrt{2} \zeta_{3}^{3} + 3 \, \sqrt{2} \zeta_{3}^{2} + \sqrt{2}\right)} \sqrt{4 \, \zeta_{3} - 5} \sqrt{\zeta_{3} - 1} - 10 \, \zeta_{3}^{2} + 5 \, \zeta_{3} - 2\right)}}
,\]
\[\hspace{-12mm}\zeta_2=-\frac{52 \, \zeta_{3}^{4} - 154 \, \zeta_{3}^{3} - {\left(18 \, \sqrt{2} \zeta_{3}^{3} - 35 \, \sqrt{2} \zeta_{3}^{2} + 13 \, \sqrt{2} \zeta_{3} - 6 \, \sqrt{2}\right)} \sqrt{4 \, \zeta_{3} - 5} \sqrt{\zeta_{3} - 1} + 148 \, \zeta_{3}^{2} - 60 \, \zeta_{3} + 18}{4 \, {\left(10 \, \zeta_{3}^{4} - 7 \, \zeta_{3}^{3} - {\left(4 \, \sqrt{2} \zeta_{3}^{3} + 3 \, \sqrt{2} \zeta_{3}^{2} + \sqrt{2}\right)} \sqrt{4 \, \zeta_{3} - 5} \sqrt{\zeta_{3} - 1} - 10 \, \zeta_{3}^{2} + 5 \, \zeta_{3} - 2\right)}},\]
\[\zeta_4=-\frac{\sqrt{2} \sqrt{4 \, \zeta_{3} - 5} \sqrt{\zeta_{3} - 1} \zeta_{3}^{2} - 6 \, \zeta_{3}^{3} + 3 \, \zeta_{3}^{2} + 3 \, \zeta_{3}}{2 \, {\left(\zeta_{3}^{2} - 1\right)}},\]
\[\zeta_5=-\frac{\sqrt{2} \sqrt{4 \, \zeta_{3} - 5} \sqrt{\zeta_{3} - 1} \zeta_{3} - 3 \, \zeta_{3}^{2} + 3 \, \zeta_{3}}{2 \, {\left(\zeta_{3}^{2} - 1\right)}},\]
assuming $\zeta_3\ge 5/4$ so that all square roots are well defined, and $\zeta_3$ large enough so that all denominators are non-zero (actually, $\zeta_3\ge 5/4$ is sufficient for both conditions).
Let $\xi_1=1$, $\xi_2=\zeta_2+\zeta_5$, $\xi_3=-\zeta_1$, $\xi_4=1+\zeta_1+\zeta_5$, $\xi_5=\zeta_2$, $\xi_6=2\zeta_2-\zeta_5$, $\xi_7=-\zeta_3$, $\xi_8=\zeta_3$, $\xi_9=\zeta_4$, $\xi_{10}=\zeta_5$ and finally $\xi=(\xi_1,\ldots,\xi_{10})$. The following lemma allows us to construct the subspace of $\R^5$ wanted.
\begin{lemma}\label{lemme_R5_existence_baeoirnfoeivnd}
There exists a subspace $A_\xi$ of dimension $3$ of $\R^5$ which admits the vector $\xi$ as Pl\"ucker coordinates (with respect to lexicographic order).
\end{lemma}
Now that the subspace $A_\xi$ has been constructed, we can state that it is indeed $(2,1)$-irrational and not so well approximated by rational planes of $\R^5$.
\begin{lemma}\label{lemme_crucial_R5_jbnaabebdbfnodsi}
There exist reals numbers $\zeta_3\ge 5/4$ and $c>0$ such that $A_\xi\in\mathfrak I_5(3,2)_1$ and
\begin{equation}\label{lemme_crucial_R5_aenpridnb}
\forall B\in\mathfrak R_5(2),\quad \varphi(A_\xi,B)\ge\frac c{H(B)^{6}}.
\end{equation}
\end{lemma}
This lemma together with Lemma \ref{minoration_psiphi_eomivocvbn} immediately leads to the following proposition. 
\begin{proposition}\label{proposition_centrale_R5_bnaeofibvds}
There exists $\zeta_3\ge 5/4$ such that 
\[\muexpA 5{A_\xi}21\le 6.\]
\end{proposition}
Similarly as in Section \ref{section_R4_boaeinvoin}, Theorem \ref{theoreme_central_R5_baneiofbonvnss} is an immediate consequence of Proposition \ref{proposition_centrale_R5_bnaeofibvds}, which itself follows from Lemma \ref{minoration_psiphi_eomivocvbn} and Lemma \ref{lemme_crucial_R5_jbnaabebdbfnodsi}. We will start with the proof of Lemma \ref{lemme_R5_existence_baeoirnfoeivnd}.
\begin{proof}[Proof of Lemma \ref{lemme_R5_existence_baeoirnfoeivnd}.]
There exists a subspace which admits $\xi$ as Pl\"ucker coordinates if, and only if, the coordinates of $\xi$ satisfy the Pl\"ucker relations (see \cite{caldero15}, Theorem 2.9) for a subspace of dimension $3$ of $\R^5$:
\begin{equation}\label{relationPlucker5_mroegomgme}
\begin{cases}
	\xi_2\xi_5=\xi_3\xi_4+\xi_1\xi_6 \\
	\xi_2\xi_8=\xi_3\xi_7+\xi_1\xi_9 \\
	\xi_4\xi_8=\xi_5\xi_7+\xi_1\xi_{10} \\
	\xi_4\xi_9=\xi_6\xi_7+\xi_2\xi_{10} \\
	\xi_5\xi_9=\xi_6\xi_8+\xi_3\xi_{10}.
\end{cases}
\end{equation}
A basic formal computation shows that the vector $\xi$ -- as it has been defined -- indeed satisfies System \eqref{relationPlucker5_mroegomgme}.
\end{proof}
Before proving the crucial Lemma \ref{lemme_crucial_R5_jbnaabebdbfnodsi}, we need a technical result.
\begin{lemma}\label{lemme_var_non_degene_oiaernboinv}
The manifold $\mathcal M=\{(1,\zeta_1,\zeta_2,\zeta_3,\zeta_4,\zeta_5),\ \zeta_3\ge 5/4\}$ is nondegenerate.
\end{lemma}
\begin{proof}
Let $(a_0,\ldots,a_5)\in\R^6$ such that $a_0+a_1\zeta_1+\cdots+a_5\zeta_5=0$ for any $\zeta_3\ge 5/4$. One can compute polynomials $P_1,P_2,P_3\in\R[X]$ such that:
\[0=a_0+a_1\zeta_1+\cdots+a_5\zeta_5=\frac{P_1(\zeta_3)+P_2(\zeta_3)\sqrt{P_3(\zeta_3)}}{10\zeta_3^3+7\zeta_3-2-(4\zeta_3^2-\zeta_3+1)\sqrt{P_3(\zeta_3)}}.\]
Hence, one has $P_1(\zeta_3)+P_2(\zeta_3)\sqrt{P_3(\zeta_3)}=0$, so for all $\zeta_3\ge 5/4$: $P(\zeta_3)=P_1^2(\zeta_3)-P_2^2(\zeta_3)P_3(\zeta_3)=0$. The four equations given by the monomials of degrees $32$, $30$, $28$ and $26$ lead to a system of equations between the $a_i$, which implies $a_0=a_3=a_4=a_5$. Considering the monomial of degree $22$ leads to $14a_1^2+4a_1a_2-a_2^2=0$, so $a_2=(2\pm 3\sqrt 2)a_1$, and the monomials of degree $21$ leads to $7a_1^2-118a_1a_2+19a_2^2=0$ which can not be. Therefore, $a_i=0$ for all $i\in\{0,\ldots,5\}$ so the manifold considered is nondegenerate.
\end{proof}
With Lemma \ref{lemme_var_non_degene_oiaernboinv}, we are now able to prove Lemma \ref{lemme_crucial_R5_jbnaabebdbfnodsi}. Notice that the proof is quite similar to the proof of Lemma \ref{condition_dirrationalite_4221_regouzhelfgu}.
\begin{proof}[Proof of Lemma \ref{lemme_crucial_R5_jbnaabebdbfnodsi}.]
Let $B\in\mathfrak R_5(2)$ and $(Y_1,Y_2)$ be a basis of $B$ provided by Lemma \ref{lemme_pour_coordPluck_premsentreelles_egouzbelvc}. Let us denote by $(\eta_1,\ldots,\eta_{10})$ a set of Pl\"ucker coordinates for $B$ associated with the basis $(Y_1,Y_2)$ ordered by lexicographic order. According to Lemma \ref{lemme_pour_coordPluck_premsentreelles_egouzbelvc}, we may assume that $(\eta_1,\ldots,\eta_{10})\in\Z^{10}$ and $\gcd(\eta_1,\ldots,\eta_{10})=1$. Moreover, this vector satisfies the Pl\"ucker relations for a subspace of dimension $2$ of $\R^5$:
\begin{equation}\label{relation_Pluck_eta_R5_aomefihv}
\begin{cases} \eta_2\eta_5=\eta_3\eta_4+\eta_1\eta_6 \\\eta_2\eta_8=\eta_3\eta_7+\eta_1\eta_9 \\ \eta_4\eta_8=\eta_5\eta_{7}+\eta_1\eta_{10} \\ \eta_4\eta_9=\eta_6\eta_{7}+\eta_2\eta_{10} \\ \eta_5\eta_9=\eta_6\eta_{8}+\eta_3\eta_{10}.\end{cases}
\end{equation}
According to Lemma \ref{lemme_var_non_degene_oiaernboinv}, the manifold $\mathcal M=\{(1,\zeta_1,\zeta_2,\zeta_3,\zeta_4,\zeta_5),\ \zeta_3\ge 5/4\}$ is nondegenerate, so Theorem \ref{th_bernesnevich_bad_pabeinbfpisnv} implies the existence of $\zeta_3\ge 5/4$ such that $(1,\zeta_1,\zeta_2,\zeta_3,\zeta_4,\zeta_5)\in\mathbf{Bad}$. In particular, $1,\zeta_1,\zeta_2,\zeta_3,\zeta_4,\zeta_5$ are linearly independent over $\Q$. Let $(X_\xi^{(1)},X_\xi^{(2)},X_\xi^{(3)})$ be a basis of $A_\xi$ associated with $\xi$. Let us denote by $M_\xi$ the matrix of $\M_5(\R)$ whose columns are $X_\xi^{(1)},X_\xi^{(2)},X_\xi^{(3)},Y_1,Y_2$ respectively. Notice that $A_\xi\cap B=\{0\}$ if, and only if, $\det M_\xi\ne 0$. The determinant of $M_\xi$ is computed by a Laplace expansion on its first three columns: 
\[\det M_\xi=\xi_1\eta_{10}-\xi_2\eta_9+\xi_3\eta_8+\xi_4\eta_7-\xi_5\eta_6+\xi_6\eta_5-\xi_7\eta_4+\xi_8\eta_3-\xi_9\eta_2+\xi_{10}\eta_1.\]
Let us assume that $\det M_\xi=0$, this implies
\begin{equation*}
\begin{split}
0
	&=\det (M_\xi) \\
	&=\eta_{10}-(\zeta_2+\zeta_5)\eta_9-\zeta_1\eta_8+(1+\zeta_1+\zeta_5)\eta_7-\zeta_2\eta_6+(2\zeta_2-\zeta_5)\eta_5+\zeta_3\eta_4+\zeta_3\eta_3-\zeta_4\eta_2+\zeta_5\eta_1 \\
	&=\eta_{10}+\eta_7+(-\eta_8+\eta_7)\zeta_1+(-\eta_9-\eta_6+2\eta_5)\zeta_2+(\eta_4+\eta_3)\zeta_3-\eta_2\zeta_4+(-\eta_9+\eta_7-\eta_5+\eta_1)\zeta_5.
\end{split}
\end{equation*}
Since $1,\zeta_1,\zeta_2,\zeta_3,\zeta_4,\zeta_5$ are linearly independent over $\Q$ and the $\eta_i$ are integers, the equation above yields the following relations:
\[(\eta_1,\eta_2,\eta_4,\eta_6,\eta_8,\eta_{10})=(\eta_9-\eta_7+\eta_5,0,-\eta_3,-\eta_9+2\eta_5,\eta_7,-\eta_7).\]
Thus, System \eqref{relation_Pluck_eta_R5_aomefihv} becomes
\begin{equation}\label{eq_x2468_aqzsedr}
\begin{cases}
	\eta_3^2-2\eta_5^2+2\eta_5\eta_7-\eta_5\eta_9-\eta_7\eta_9+\eta_9^2=0 \\
	-\eta_3\eta_7-\eta_5\eta_9+\eta_7\eta_9-\eta_9^2=0 \\
	-\eta_3\eta_7-\eta_7^2+\eta_7\eta_9=0 \\
	-2\eta_5\eta_7-\eta_3\eta_9+\eta_7\eta_9=0 \\
	\eta_3\eta_7-2\eta_5\eta_7+\eta_5\eta_9+\eta_7\eta_9=0
\end{cases}
\end{equation}
whose set of rational solutions is the singleton $\{(0,\ldots,0)\}$ (once again, the computations can be found in \cite{joseph21}). Thereby, $\det M_\xi\ne 0$, so $A_\xi\cap B=\{0\}$ which implies that $A_\xi\in\mathfrak I_5(3,2)_1$. \\

The proof of second part of the lemma is almost identical as the proof of \eqref{minoration_hauteur_lemme_R4_zidovdbon} in Lemma \ref{condition_dirrationalite_4221_regouzhelfgu}, but with $6$ reals numbers instead of $3$.
\end{proof}
\begin{remark}
Similarly as in Section \ref{section_R4_boaeinvoin}, one can construct infinitely many subspaces $A_\xi$ defined over $\overline\Q$ satisfying $\muexpA 5{A_\xi}21\le 6$ with Theorem 2 of \cite{schmidt70}. The only difference is that the exponent $-6$ in Equation \eqref{lemme_crucial_R5_aenpridnb} becomes $-6-\epsilon$ for any $\epsilon>0$. Up to this modification, Lemma \ref{lemme_crucial_R5_jbnaabebdbfnodsi} and Proposition \ref{proposition_centrale_R5_bnaeofibvds} are still true if $\zeta_3\ge 5/4$ is a real algebraic number satisfying $[\Q(\zeta_3):\Q]\ge 33$. 
\end{remark}

\section{Some comments on the method}\label{section_comments_vaoebnisfon}

We believe that the method developed in Sections \ref{section_R4_boaeinvoin} and \ref{section_R5_epribnrfv} can be used to improve several other upper bounds for $\muexp nde1$ when $d+e=n$. As one can see in Section \ref{section_R5_epribnrfv}, the computations seem to be significantly more complicated with $n$ growing. The main difficulty in $\R^5$ was to construct a subspace $A_\xi$ complicated enough so that System \eqref{eq_x2468_aqzsedr} would not have any non trivial rational solution -- which implies $A_\xi\in\mathfrak I_5(3,2)_1$ -- but also sufficiently simple so that it is indeed possible to show that this system does not have any non trivial rational solution. 

This method creates two contradictory wishes on the subspace $A$ desired:
\begin{itemize}
	\item to have \emph{a lot} of Pl\"ucker coordinates linearly independent on $\Q$ so that $A$ is $(e,1)$-irrational;
	\item to have \emph{few} Pl\"ucker coordinates linearly independent on $\Q$ to obtain the best possible exponent with Theorem \ref{th_bernesnevich_bad_pabeinbfpisnv}.
\end{itemize}

\section{Application of Schmidt's Going-up theorem}\label{section_corollaire_Moshchevitin_peinvipsnvpnp}

Here, we will prove Corollary \ref{cor_goingup_induction_baioeboifv} which implies Proposition \ref{prop_app_goingup_murond_oebonio} from which is immediately deduced Theorem \ref{theoreme_avec_moshchevitin_goingup_aeronfvn}: $\muexp {2d}{d}{\ell}1\le 2d^2/(2d-\ell)$. Indeed, Proposition \ref{prop_app_goingup_murond_oebonio} together with Theorem \ref{theoreme_moshchevitin_abenobvin} gives for $\ell\in\{1,\ldots,d\}$: $\muexp{2d}d\ell1\le (2d-d)/(2d-\ell)\muexp{2d}dd1\le 2d^2/(2d-\ell)$.

Theorem \ref{theoreme_avec_moshchevitin_goingup_aeronfvn} allows us to improve on numerous known upper bounds for $\muexp {2d}d{\ell}1$, since for instance taking $\ell=d-1$ implies
\[\frac{2d^2}{2d-\ell}\underset{d\to+\infty}\sim2d\]
and the known upper bound for $\muexp {2d}d{d-1}1$, given by Theorem \ref{encadrement_general_schmidt_aoeibinvn}, is asymptotically equivalent to $\lfloor d^2/2\rfloor$. Notice that when $\ell$ is fixed and $d$ tends to $+\infty$, Theorem \ref{encadrement_general_schmidt_aoeibinvn} gives an upper bound asymptotically equivalent to $2\ell$, which is better than our new bound. The best improvements occur when $\ell$ is close to $d$, for instance Theorem \ref{theoreme_avec_moshchevitin_goingup_aeronfvn} implies $\muexp 6321\le 9/2$ improving on $\muexp 6321\le 5$, $\muexp {12}641\le 9$ improving on $\muexp {12}641\le 11$, and $\muexp{22}{11}61\le 15.125$ improving on $\muexp {12}641\le 17$.

Let us now state Schmidt's Going-up theorem (see \cite{schmidt67}, Theorem 9).
\begin{theorem}[Going-up, Schmidt, 1967]\label{goingup_eorihgzmefvnfon}
Let $d,e\in\N^*$ be such that $d+e<n$; let $t=\min(d,e)$. Let $A$ be a subspace of $\R^n$ of dimension $d$ and $B\in\mathfrak R_n(e)$. Let $H\ge 1$ be such that $H(B)\le H$, and such that there exist $x_i,y_i\in\R$ such that for all $i\in\{1,\ldots,t\}$, $H(B)^{x_i}\psi_i(A,B)\le c_1H^{-y_i}$ with $c_1>0$. Then there exists a constant $c_2>0$ depending only on $n$ and $e$, and a constant $c_3>0$ depending only of $n$, $e$, $x_i$ and $y_i$, such that if $H'=c_2H^{(n-e-1)/(n-e)}$, then there exists $C\in\mathfrak R_n(e+1)$ such that $C\supset B$, $H(C)\le H'$ and
\[\forall i\in\{1,\ldots,t\},\quad H(C)^{x_i(n-e)/(n-e-1)}\psi_i(A,C)\le c_1c_3H'^{-y_i(n-e)/(n-e-1)}.\]
\end{theorem}
Let us formulate a corollary to the Going-up theorem.
\begin{corollary}\label{cor_goingup_induction_baioeboifv}
Let $d,e,j,\ell\in\N^*$ be such that $d+e\le n$, $1\le j\le \ell\le e$ and $j\le d$. Then for all $A\in\mathfrak I_n(d,e)_j$, one has $A\in\mathfrak I_n(d,\ell)_j$ and
\[\muexpA nAej\ge \frac{n-\ell}{n-e}\cdot\muexpA nA\ell j.\]
\end{corollary}
Since $\mathfrak I_n(d,e)_j\subset \mathfrak I_n(d,\ell)_j$, Corollary \ref{cor_goingup_induction_baioeboifv} implies immediately Proposition \ref{prop_app_goingup_murond_oebonio} stated in the introduction.
\begin{remark}
Notice that Corollary \ref{cor_goingup_induction_baioeboifv} generalises Theorem 2 of \cite{laurent09}. Corollary \ref{cor_goingup_induction_baioeboifv} does not necessarily need to be applied on a line, and the irrationality hypothesis is weaker than the one in \cite{laurent09}.
\end{remark}
\begin{proof}[Proof of Corollary \ref{cor_goingup_induction_baioeboifv}.]
Notice that $\mathfrak I_n(d,e)_j\subset \mathfrak I_n(d,\ell)_j$ since $\ell\le e$. Let $\alpha=\muexpA n{A}{\ell}j$ and $\epsilon>0$; there exist infinitely many subspaces $B\in\mathfrak R_n(\ell)$ such that
\begin{equation}\label{ineg_psi_1_A_B_arpeibnapifvndpiva}
\psi_j(A,B)\le \frac 1{H(B)^{\alpha-\epsilon}}.
\end{equation}
For each such subspace $B$, the Going-up theorem applied $e-\ell$ times gives a subspace $C\in\mathfrak R_n(e)$ such that $C\supset B$ and
\begin{equation}\label{preuve6321_erjgmoiefhgnoeinv}
\psi_j(A,C)\le\frac {c}{H(C)^{(\alpha-\epsilon)(n-\ell)/(n-e)}}
\end{equation}
with $c>0$ depending only on $A$ and $\epsilon$. The subspace $A$ is $(e,j)$-irrational, so for all $C\in\mathfrak R_n(e)$, $\psi_j(A,C)\ne 0$. Thus, if there were only a finite number of rational subspaces $C$ such that Inequality \eqref{preuve6321_erjgmoiefhgnoeinv} holds, there would be a constant $c'>0$ such that 
\begin{equation}\label{minoration_psi1_A_eps_C_apihnramoegbobb}
\forall C\in\mathfrak R_n(e),\quad \psi_j(A,C)>c'.
\end{equation}
Since there are infinitely many subspaces $B\in\mathfrak R_n(\ell)$ such that Inequality \eqref{ineg_psi_1_A_B_arpeibnapifvndpiva} holds, there exist such subspaces of arbitrary large height, thus such that $\psi_j(A,B)\le c'$. The subspace $C$ obtained from $B$ with the Going-up theorem satisfies $B\subset C$, so $\psi_j(A,C)\le \psi_j(A, B)\le c'$, which contradicts \eqref{minoration_psi1_A_eps_C_apihnramoegbobb}. Hence, there are infinitely many subspaces $C\in\mathfrak R_n(e)$ such that \eqref{preuve6321_erjgmoiefhgnoeinv} holds, and the corollary follows.
\end{proof}

\section{A lower bound for $\muexp ndej$ in the general case}\label{section_lowerbound_general_paifnvpvisnv}

The goal here is to prove a new lower bound for $\muexp ndej$ (Theorem \ref{premiereborneobtenue_amorimeovbn}). The strategy is to break down the subspace we want to approach into subspaces of lower dimension (here, we will use lines). It is then possible to approach simultaneously each line (it will be done with Dirichlet's approximation theorem), and to deduce an approximation of the original subspace.

The bound given by Theorem \ref{premiereborneobtenue_amorimeovbn} improves asymptotically (for fixed $j$, $d$ and $e$) the known lower bound for $\muexp ndej$ (Theorem \ref{encadrement_general_schmidt_aoeibinvn}). 

Let $d\le n/2$. Combining Theorem \ref{premiereborneobtenue_amorimeovbn} with Theorem \ref{theoreme_saxcé_piebpisnvpzin}, one obtains
\[\frac{2dn-d^2+d+2}{2d^2n-d^3+d^2}\le \muexp nddd\le \frac{n}{d(n-d)},\]
hence Corollary \ref{corollaire_th_lower_bound_etdeSaxce_zpiznvpsinvc}:
\[\lim_{n\to+\infty} \muexp nddd=\frac 1d.\]

The proof of Theorem \ref{premiereborneobtenue_amorimeovbn} will require a lemma on the behaviour of the proximity function $\psi$ with direct sums. 
\begin{lemma}\label{resultat_reconstruction_proximite_oaeirgbvodbv}
Let $n\ge 4$ and $F_1,\ldots,F_\ell, B_1,\ldots,B_\ell$ be $2\ell$ subspaces of $\R^n$ such that for all $i\in\{1,\ldots,\ell\}$, $\dim F_i=\dim B_i=d_i$. Assume that the $F_i$ span a subspace of dimension $k=d_1+\cdots+d_\ell$ and so do the $B_i$. Let $F=F_1\oplus\cdots\oplus F_\ell$ and $B=B_1\oplus\cdots\oplus B_\ell$, then one has
\[\psi_k(F,B)\le c_{F,n}\sum_{i=1}^\ell\psi_{d_i}(F_i,B_i)\]
where $c_{F,n}>0$ is a constant depending only on $F_1,\ldots,F_\ell$ and $n$.
\end{lemma}
\begin{proof}
The idea is to break down each $F_i$ and each $B_i$ into a direct sum of well chosen lines. For this, we will use the following claim.
\begin{claim}\label{minorationinverseavecdroitesbienchoisies_paeivmsoivn}
Let $D$ and $E$ be two subspaces of $\R^n$ of dimension $k$. There exist $k$ lines $D_1,\ldots,D_k$ of $D$ and $k$ lines $E_1,\ldots,E_k$ of $E$, such that $D=D_1\oplus\cdots\oplus D_k$, $E=E_1\oplus\cdots\oplus E_k$, and
\begin{equation}\label{inegalite_claim_41_apegnspinv}
\psi_k(D,E)\le\sum_{i=1}^k \psi_1(D_i,E_i)\le k\psi_k(D,E).
\end{equation}
\end{claim}
\begin{proof}[Proof of Claim \ref{minorationinverseavecdroitesbienchoisies_paeivmsoivn}.]
There exist an orthonormal basis $(X_1,\ldots,X_k)$ of $D$ and an orthonormal basis $(Y_1,\ldots,Y_k)$ of $E$ such that for all $i\in\{1,\ldots,k\}$, $\psi_i(D,E)=\psi(X_i,Y_i)$. Moreover, for all $i\in\{1,\ldots,k\}$, one has $\psi_i(D,E)\le \psi_k(D,E)$. Let us denote for $i\in\{1,\ldots,k\}$, $D_i=\span(X_i)$ and $E_i=\span(Y_i)$ to get the second part of Inequality \eqref{inegalite_claim_41_apegnspinv}:
\[\sum_{i=1}^k\psi_1(D_i,E_i)=\sum_{i=1}^k \psi(X_i,Y_i)=\sum_{i=1}^k \psi_i(D,E)\le k\psi_k(D,E).\]
The first part of Inequality \eqref{inegalite_claim_41_apegnspinv} is trivial since $\psi_1(D_i,E_i)\ge 0$ for any $i$, and $\psi_k(D,E)=\psi_1(D_k,E_k)$.
\end{proof}
We can come back to the proof of Lemma \ref{resultat_reconstruction_proximite_oaeirgbvodbv}. Let $i\in\{1,\ldots,\ell\}$; according to Claim \ref{minorationinverseavecdroitesbienchoisies_paeivmsoivn}, there exist $d_i$ lines $D_{i,1},\ldots,D_{i,d_i}$ of $F_i$ and $d_i$ lines $E_{i,1},\ldots,E_{i,d_i}$ of $B_i$ such that
\begin{equation}\label{majorationinvgraceaulemme}
\sum_{j=1}^{d_i} \psi_1(E_{i,j},D_{i,j})\le d_i\psi_{d_i}(F_i,B_i) \le n\psi_{d_i}(F_i,B_i).
\end{equation}
Let $a_{i,1},\ldots,a_{i,d_i}$ be unitary vectors of $D_{i,1},\ldots,D_{i,d_i}$ respectively and $b_{i,1},\ldots,b_{i,d_i}$ be unitary vectors of $E_{i,1},\ldots,E_{i,d_i}$ respectively, such that for all $j\in\{1,\ldots,d_i\}$, $a_{i,j}\cdot b_{i,j}\ge 0$. Let $(X_1,\ldots,X_k)$ and $(Y_1,\ldots,Y_k)$ be orthonormal bases of $F$ and $B$ respectively, such that $\psi_j(F,B)=\psi(X_j,Y_j)$ for any $j\in\{1,\ldots,k\}$. Let $Z=\lambda_1 Y_1+\cdots+\lambda_k Y_k$ be a unitary vector of $B$. One has
\[\abs{X_k\cdot Z}=\abs{\sum_{i=1}^k \lambda_i X_k\cdot Y_i}\le\sum_{i=1}^k \abs{\lambda_i \delta_{i,k} X_k\cdot Y_i} \le X_k\cdot Y_k\]
which implies that
\[\psi_k(F,B)=\psi(X_k,Y_k)\le \min_{Z\in B\setminus\{0\}}\psi(X_k,Z)=\psi_1(\span(X_k),B).\]
Moreover, $\span(Y_k)\subset B$, so $\psi_1(\span(X_k),B)\le \psi(X_k,Y_k)$. Hence 
\begin{equation}\label{egalite_psi_k_aroeihnveoivn}
\psi_k(F,B)=\psi_1(\span(X_k),B).
\end{equation}
Let us decompose $X_k$ in the basis $(a_{1,1},\ldots,a_{\ell,d_\ell})$: $X_k=\sum_{i=1}^\ell\sum_{j=1}^{d_i} x_{i,j} a_{i,j}$, and let
\[Y=\sum_{i=1}^\ell\sum_{j=1}^{d_i} x_{i,j}b_{i,j}\in B.\]
Since $X_k$ is unitary, one has
\[\psi(X_k,Y)\le \norme{X_k-Y}=\norme{\sum_{i=1}^\ell\sum_{j=1}^{d_i} x_{i,j} (a_{i,j}-b_{i,j})}\le \sum_{i=1}^\ell\sum_{j=1}^{d_i} \abs{x_{i,j}} \norme{a_{i,j}-b_{i,j}},\]
where $\norme\cdot$ stands for the Euclidean norm. For $i\in\{1,\ldots,\ell\}$ and $j\in\{1,\ldots,d_i\}$, let us consider the functions 
\[\fonction{p_{i,j}} F \R {\displaystyle\sum_{i=1}^\ell\sum_{j=1}^{d_i}x_{i,j}a_{i,j}}{x_{i,j}.}\]
These functions are continuous on the compact $K=\{x\in F,\ \norme x=1\}$, so they are bounded on it. Thus, there exists $c_{F,n}^{(1)}$ a constant depending only on $a_{1,1},\ldots,a_{\ell,d_\ell}$ such that for all $x=\sum_{i=1}^\ell\sum_{j=1}^{d_i} x_{i,j} a_{i,j}\in K$, one has $\abs{x_{i,j}}\le c_{F,n}^{(1)}$.

We now require an elementary claim.
\begin{claim}\label{lemme_minoration_psiXY_amoifbvoisbdvoibs}
Let $X$ and $Y$ be unitary vectors such that $X\cdot Y\ge 0$. One has
\[\psi(X,Y)\ge \frac{\sqrt 2}{2} \norme{X-Y}.\]
\end{claim}
\begin{proof}
Let $p_{\span(Y)}^\perp$ be the orthogonal projection onto $\span(Y)$, $\alpha=\Vert X-p_{\span(Y)}^\perp(X)\Vert$ and $\beta=\Vert Y-p_{\span(Y)}^\perp(X)\Vert$. One has $\norme{X-Y}^2=\alpha^2+\beta^2$, and since $X$ is unitary: $\psi(X,Y)=\psi(X,p_{\span(Y)}^\perp(X))=\Vert X-p_{\span(Y)}^\perp(X)\Vert=\alpha$. Moreover, $X\cdot Y\ge 0$, so $1=\norme{X}^2=(1-\beta)^2+\alpha^2$, hence there exists $\theta\in[0,\pi/2]$ such that $1-\beta=\cos\theta$ and $\alpha=\sin\theta$. Since $1-\cos\theta\le \sin\theta$, yields $\beta\le \alpha$, and finally $\norme{X-Y}^2\le 2\alpha^2=2\psi(X,Y)^2$.
\end{proof}

We can come back to the proof of Lemma \ref{resultat_reconstruction_proximite_oaeirgbvodbv}. Since for all $i,j$ one has $a_{i,j}\cdot b_{i,j}\ge 0$, applying Claim \ref{lemme_minoration_psiXY_amoifbvoisbdvoibs} yields to
\[\psi(X_k,Y)\le c_{F,n}^{(1)} \sum_{i=1}^\ell\sum_{j=1}^{d_i} \norme{a_{i,j}-b_{i,j}}\le c_{F,n}^{(2)} \sum_{i=1}^\ell\sum_{j=1}^{d_i} \psi_1(D_{i,j},E_{i,j})\]
because the $a_{i,j}$ and the $b_{i,j}$ are unitary vectors, with $c_{F,n}^{(2)}=\sqrt 2c_{F,n}^{(1)}$. Finally, Inequality \eqref{majorationinvgraceaulemme} implies
\begin{equation}\label{majoration_psi_preuve_prox_amoefivbv}
\psi(X_k,Y)\le c_{F,n}^{(2)}n \sum_{i=1}^\ell \psi_{d_i}(F_i,B_i)
\end{equation}
and with Equation \eqref{egalite_psi_k_aroeihnveoivn} yields
\[\psi_k(F,B)\le\psi_1(\span(X_k),B)\le \psi(X_k,Y)\]
because $Y\in B$. Using Inequality \eqref{majoration_psi_preuve_prox_amoefivbv}, it follows
\[\psi_k(F,B)\le c_{F,n} \sum_{i=1}^\ell \psi_{d_i}(F_i,B_i).\]
\end{proof}
Now that Lemma \ref{resultat_reconstruction_proximite_oaeirgbvodbv} is proved, we can tackle the proof of Theorem \ref{premiereborneobtenue_amorimeovbn}.
\begin{proof}[Proof of Theorem \ref{premiereborneobtenue_amorimeovbn}.]
Let $F\in\mathfrak I_n(d,e)_j$. Let us show by induction that $F$ possesses an orthonormal family $(f_1,\ldots,f_j)$ such that for all $\ell\in\{1,\ldots,j\}$, at least $d-\ell$ coordinates of $f_\ell$ vanish. For $\ell=0$ there is nothing to show; let us assume that $f_1,\ldots,f_\ell$ have been constructed. Let us denote by $G$ the orthogonal complement of $\span(f_1,\ldots,f_\ell)$ in $F$. One has $G\cap(\R^{n-d+\ell+1}\times\{0\}^{d-\ell-1})\ne \{0\}$ because $\mathrm{codim}(\R^{n-d+\ell+1}\times\{0\}^{d-\ell-1})=\dim G-1$, let $f_{\ell+1}\in G\cap(\R^{n-d+\ell+1}\times\{0\}^{d-\ell-1})$ be a unitary vector. At least $d-(\ell+1)$ coordinates of this vector vanish, and it is orthogonal to $f_1,\ldots,f_\ell$. 

In all what follows, let $(f_1,\ldots,f_j)$ be an orthonormal family of $F$ such that for all $\ell\in\{1,\ldots,j\}$, at least $d-\ell$ coordinates of $f_\ell$ vanish. Let us denote by $\underline x$ the vector formed with all the non-zero coordinates of the $f_\ell$ and denote by $N\in\{1,\ldots,jn-jd+j^2/2+j/2\}$ its number of coordinates. 

One has $\underline x\in\R^N\setminus\Q^N$, otherwise $(f_1,\ldots,f_j)$ would span a subspace of dimension $j$ of $F$, which can not be since $F\in\mathfrak I_n(d,e)_j$. Using Dirichlet's approximation theorem, there exist infinitely many couples $(p,q)\in\Z^N\times \N^*$ such that $\gcd(p_1,\ldots,p_N,q)=1$ and 
\begin{equation}\label{approxsimultaneedeDir}
\norme{\underline x-\frac pq}_{\infty}\le \frac 1{q^{1+1/N}}.
\end{equation}
Let us fix such a couple $(p,q)$. For $i\in\{1,\ldots,j\}$, let us denote by $p_i$ the subfamily of $p$ corresponding to its coordinates approaching those of $f_i$, completed with zeros so that $p_i\in\Z^n$ is close to $qf_i$. For all $i\in\{1,\ldots,j\}$, one has $\Vert f_i-p_i/q\Vert_{\infty}\le q^{-1-1/N}$.

Let $B=\span(p_1,\ldots,p_j)$, and let us denote by $p_i^\perp(f_i)$ the orthogonal projection of $f_i$ onto $\span(p_i/q)$. One has
\begin{equation}\label{majoration_psi_moaeirbsovmin}
\psi(f_i,p_i/q)=\sin\widehat{\left(f_i, p_i/q\right)}=\frac{\norme{f_i-p_i^\perp(f_i)}}{\norme{f_i}}\le\norme{f_i-\frac{p_i}q}\le \frac{c_1}{q^{1+1/N}}
\end{equation}
because $\norme{f_i}= 1$, with $c_1>0$ depending only on $n$. Inequality \eqref{approxsimultaneedeDir} gives $\norme{p}_\infty-\norme{q\underline x}_\infty\le \norme{q\underline x-p}_\infty\le q^{-1/N}\le 1$, so for all $i\in\{1,\ldots,j\}$: $\norme{p_i}_\infty\le \norme{p}_\infty\le 1+\norme{q\underline x}_\infty\le c_2q$, with $c_2>0$ depending only on $F$. 

For $E$ a subspace of $\R^n$ and $P$ a family of linearly independent vectors of $E$, let us denote by $\mathrm{vol}_E(P)$ the volume of the parallelotope spanned by the vectors of $P$ and considered in the Euclidean space $E$. Since $(p_1,\ldots,p_j)$ is a sublattice of $B\cap\Z^n$, one has using Theorem \ref{th_def_equiv_hauteur_vaoribibgipn}:
\[H(B)\le\mathrm{vol}_B(p_1,\ldots,p_j)\le \prod_{i=1}^j \norme{p_i}\le c_3 q^j\]
with $c_3>0$ depending only on $F$. Thus, there exists a constant $c_4>0$ such that
\begin{equation}\label{majoration_1q_aprighmoidvn}
\frac 1q\le \frac{c_4}{H(B)^{1/j}}.
\end{equation}
Let $\tilde F_j=\span(f_1,\ldots,f_j)$ which is a subspace of dimension $j$ of $F$, and let $B_i=\span(p_i)$ for $i\in\{1,\ldots,j\}$. According to Proposition \ref{resultat_reconstruction_proximite_oaeirgbvodbv} and Inequality \eqref{majoration_psi_moaeirbsovmin}, one has
\begin{equation}\label{majoration_psi_j_aomerifnvmeoinv}
\psi_j(\tilde F_j,B)=\psi_j\left(\bigoplus_{i=1}^j\span(f_i),\bigoplus_{i=1}^jB_i\right)\le c_5 \sum_{i=1}^j  \psi_1(\span(f_i),B_i)\le \frac{c_6}{q^{(N+1)/N}}
\end{equation}
with $c_5,c_6>0$ depending only on $n$ and $F$. Moreover, $F\supset \tilde F_j$, so $\psi_j(F,B)\le \psi_j(\tilde F_j,B)$. Thus, Inequalities \eqref{majoration_1q_aprighmoidvn} and \eqref{majoration_psi_j_aomerifnvmeoinv} show that there exists a constant $c_7>0$ depending only on $n$ and $F$ such that
\begin{equation}\label{majoration_psij_aefombrriehv}
\psi_j(F,B)\le \frac{c_7}{H(B)^{(N+1)/(jN)}}\le \frac{c_{7}}{H(B)^{(jn-jd+j^2/2+j/2+1)/(j(jn-jd+j^2/2+j/2))}},
\end{equation}
hence 
\[\muexp ndjj\ge \frac{jn-jd+j^2/2+j/2+1}{j^2(n-d+j/2+1/2)}\]
and the result follows from Proposition \ref{prop_app_goingup_murond_oebonio}.
\end{proof}

\bibliographystyle{alpha}
\bibliography{biblio_article_1}

\end{document}